\documentclass[12pt,centertags,oneside, reqno]{amsart}
\usepackage{ifpdf}

\usepackage{amssymb}
\usepackage{braket}
\usepackage[all]{xy}
\usepackage{cancel}

\usepackage{amscd,amsxtra,calc}
\usepackage{cmmib57}
\usepackage{url}
\usepackage{ulem}
\usepackage{xcolor}

\usepackage{tikz}
\usepackage{array, longtable}

\usepackage[a4paper,width=16.2cm,top=3cm,bottom=3cm]{geometry}

\numberwithin{equation}{section}

\theoremstyle{plain}
\newtheorem{thm}{Theorem}[section]
\newtheorem{theorem}[thm]{Theorem}

\newtheorem{lemma}[thm]{Lemma}
\newtheorem{corollary}[thm]{Corollary}
\newtheorem{proposition}[thm]{Proposition}
\theoremstyle{definition}
\newtheorem{remark}[thm]{Remark}

\numberwithin{equation}{section}

\newcommand{\sS}{{\mathcal S}}

\newcommand{\C}{{\mathbb C}}

\newcommand{\F}{{\mathbb F}}

\renewcommand{\H}{{\mathbb H}}

\renewcommand{\P}{{\mathbb P}}
\newcommand{\Q}{{\mathbb Q}}
\newcommand{\R}{{\mathbb R}}

\newcommand{\Z}{{\mathbb Z}}

\newcommand{\Aut}{{\rm Aut\hspace{.1ex}}}

\title[Fermat K3]{Sixteen generators of the automorphism group of the Fermat quartic surface}

\author{JongHae Keum, Keiji Oguiso, Xun Yu}

\address{HCMC, Korea Institute for Advanced Study, Seoul 02455,
Republic of Korea}
\email{jhkeum@kias.re.kr}

\address{Mathematical Sciences, the University of Tokyo, Meguro Komaba 3-8-1,
Tokyo, Japan, and National Center for Theoretical Sciences,
Mathematics Division, National Taiwan University,
Taipei, Taiwan}
\email{oguiso@ms.u-tokyo.ac.jp}

\address{Center for Applied Mathematics and KL-AAGDM, Tianjin University, Weijin Road 92, Tianjin 300072, P.R. China}
\email{xunyu@tju.edu.cn}

\thanks{JongHae Keum is partially supported by  NRF of Korea (RS-2022-NR068993), Keiji Oguiso is partially supported by JSPS Grant-in-Aid (A) 25H00587, 25K21992 and NCTS scholar program. Xun Yu is partially supported by NSFC (No. 12071337).}
\dedicatory{}


\begin{document}

\maketitle

\begin{abstract}
It has been a long standing open problem to find generators of the automorphism group of the complex Fermat quartic K3 surface; in fact, even an explicit number of generators has not been known before. In this paper, we provide the first solution to this problem, presenting $16$ geometric generators of finite order for this group explicitly.  \end{abstract}

\section{Introduction}\label{sect0}

We work over the complex number field $\C$. By a K3 surface, we mean a smooth projective surface with no non-zero global holomorphic one form and with a nowhere vanishing global holomorphic two form.

\medskip

Let $V$ be a K3 surface. Finite groups acting on K3 surfaces are fairly well-understood (\cite{Mu88}, \cite{Ko98a}, \cite{BH23}, and so on). Based on the Torelli theorem for algebraic K3 surfaces (\cite{PS71}), Sterk \cite{St85} shows that the full automorphism group $\Aut\, (V)$ of $V$ is always finitely generated as a group. However, explicit generators are not known except for a very few cases, namely, the case where $\Aut\, (V)$ is finite or virtually cyclic (e.g. in the case where $\rho (V) \le 2$), the two most algebraic K3 surfaces $X_3$ and $X_4$ by Vinberg \cite{Vi83}, Kummer surfaces associated to generic genus two curves by Kondo \cite{Ko98b}, Kummer surfaces of certain product type by Kondo and the first author \cite{KK01}, and quartic
Hessian surfaces by Dolgachev and the first author \cite{DK02}.

\medskip

Segre \cite{Sg44} proved that the Fermat quartic surface 
$$F_4^2 := (x_0^4+ x_1^4 + x_2^4+ x_3^4 = 0) \subset \P_{\C}^3$$
has infinitely many discrete automorphisms. It was the first example of an algebraic surface with infinite discrete automorphism group. Since then, it has been a long standing open problem to find generators of the automorphism group of the Fermat quartic K3 surface $F_4^2$. (See Remark \ref{rem1} for the automorphism groups of other Fermat hypersurfaces.) Indeed, Shimada, an expert in this subject, applied an algorithm based on Borcherds method to this problem, and wrote ''the computation is very heavy, and we could not finish the calculation'' (\cite[Page 11964]{Sm15}, see also \cite[Section 5, Intractable Examples]{Sm14}).

\medskip

In this paper, we provide the first explicit solution to this problem, $16$ generators for the automorphism group, $15$ involutions and one of order $4$ (see also Remark \ref{rmk:order4}).  Our main result is the following:

\medskip

\begin{theorem}\label{main1}
The automorphism group ${\rm Aut}(F_4^2)$ is generated by the sixteen geometrically explicit automorphisms of finite order listed in Table \ref{tab:16generators}.
  \end{theorem}

We include our two weaker results (Theorems \ref{main2}, \ref{main3}).
Consider the Kummer surfaces $$V_n:={\rm Km}(E_{\sqrt{-1}}\times E_{2^{n-2}\sqrt{-1}})\,\, (n\ge 2)$$ of maximal Picard number $20$ and the most algebraic K3 surface of discriminant $4$ $$V_1=X_4.$$ Note that $V_2$ is the Kummer surface studied in \cite{KK01} and $V_3$ is the Fermat quartic surface $F_4^2$ by Shioda \cite[Example 5.2]{Sh79}. We denote by $S_{n}$ the N\'eron-Severi lattice of $V_n$ $$S_n= {\rm NS}(V_n).$$

\begin{theorem}\label{main2} Let
$$l := \frac{(2^{20}-1)(2^{19}-1)}{3} =  183251413675.$$
Then for an odd integer $n\ge 3$, $\Aut\, (V_n)$ has a finite set of generators of cardinality
$$11 \cdot (8l)^{\frac{n-1}{2}} +1.$$
In particular,  ${\rm Aut}(F_4^2)$  has a finite set of generators of cardinality
$$11 \cdot 8\cdot 183251413675 +1 = 16126124403401.$$
\end{theorem}

We define two subgroups of the isometry group ${\rm O}(S_n)$ of $S_n$
$${\rm O}'(S_n) := \{ g \in {\rm O}\, (S_n)\, |\, g(P_n) = P_n\},$$
$${\rm O}''(S_n) := \{ g \in {\rm O}'\, (S_n)\, |\, g_{A_{S_n}} = {\rm id}_{A_{S_n}}\}<{\rm O}'(S_n),$$
where $P_n\subset S_n\otimes_\Z\R$ is the positive cone of $S_n$, and $A_{S_n}$ is the discriminant group of $S_n$ (see Section \ref{sect1} for more details).

In Section \ref{sect1}, we use {\it the overlattice structure \(S_{n+2}\subset S_n\)} (see Lemma \ref{lem:isogeny}) and Vinberg's $11$-generator bound for \(S_1\) (Theorem \ref{thm11}) to inductively bound \({\rm O}''(S_n)\), via a uniform control on the index \([{\rm O}''(S_n):{\rm O}''(S_{n+2})]\) (Proposition \ref{prop14}). Combined with the fact that the image of the canonical representation ${\rm Aut}(V_n)\rightarrow {\rm GL}(H^0(V_n, \Omega_{V_n}^2))$ is of order $4$, this yields a bound \(11\cdot(8l)^{(n-1)/2}+1\) for this infinite Kummer family $V_n$. Our proof of Theorem \ref{main2} is computer-free.

\begin{theorem}\label{main3}
The group ${\rm O}'(S_3)$ is generated $19361$ explicit isometries and  the group ${\rm Aut}(F_4^2)$ is generated by $4\cdot(19361-1)+1=77441$ automorphisms.
\end{theorem}

Section \ref{sect2} develops, for a group $G$ with a given tessellation for \(G\), a constructive method to explicitly obtain a fundamental domain for a finite-index subgroup \(H\) of \(G\), and consequently provides explicit upper bounds for the minimal number of generators \(l(H)\) in terms of the fundamental domain of $H$ (Theorems \ref{thm:bound}, \ref{thm:sharperbound}). Applying these results to ${\rm O}'(S_{3})\subset {\rm O}'(S_{2})$ together with known results on $V_2$ (\cite{KK01}), we find a fundamental domain of the action of ${\rm O}'(S_{3})$ on the positive cone $$P_2=P_3 \subset S_{3} \otimes_{\Z} \R =S_{2} \otimes_{\Z} \R$$ as a union of $112$ fundamental domains of ${\rm O}'(S_{2})$ on $P_2$ and we obtain $19361$ explicit generators in Theorem \ref{main3}, which will be used in the proof of Theorem \ref{main1}. The seond statement of Theorem \ref{main3} follows from the first statement and Proposition \ref{pp:Suzuki}.

Now we briefly explain the idea of our proof of Theorem \ref{main1}. Let $S$ be the N\'eron-Severi lattice of a K3 surface $V$ containing an ample class $h$. Section \ref{sect5} introduces a practical {\it Dirichlet reduction} algorithm: for a generating set $\mathcal{S}$ of the group \({\rm O}'(S)\), one replaces an isometry \(f_i\in \mathcal{S}\) by a word \(f_w\) with a smaller intersection number \((f_w(h),h)\), iteratively lowering a total \(h\)-degree $\Sigma_{f\in \mathcal{S}}(f(h),h)$ to shrink the generating set (see Remark \ref{rmk:Dirichlet}). Applied to \(V_3=F_4^2\), starting from the $19361$ explicit generators of \({\rm O}'(S_3)\) in the proof of Theorem \ref{main3} and using the Dirichlet reduction repeatedly together with \cite[Page 180, Theorem 6.9]{Su82}, we finally obtain the $16$ automorphisms in Table \ref{tab:16generators} generating ${\rm Aut}(F_4^2)$.

Our methods of finding generators of automorphism groups are different from the existing ones such as \cite{Vi83}, \cite{Ko98b}, \cite{KK01}, \cite{DK02}, \cite{DK03}, \cite{KS14}, \cite{Sm14}, \cite{Sm15}, \cite{Sm16}.

\begin{remark}\label{rem1}
Let $F_d^n \subset \P_{\C}^{n+1}$ be the Fermat hypersurface of dimension $n \ge 2$ and degree $d \ge 3$. Then
$$\Aut (F_d^n) = (\mu_{d}^{n+2} \rtimes \mathfrak{S}_{n+2})/\mu_{d}\,\,,(n,d)\neq (2,4),$$
where $\mu_{d}$ is the group of $d$-th roots of unity and $ \mathfrak{S}_{n+2}$ is the symmetric group on $n+2$ letters,  as expected (see eg. \cite{Sh88}, also \cite{YYZ25} and \cite{EL25}), whereas $\Aut (F_4^2)$, or more generally $\Aut (V)$ of a K3 surface $V$ with $\rho(V) =20$, is an infinite group containing $\Z * \Z$ (see \cite{SI77}, \cite{Og07} and \cite{Yu25}). See \cite{Og05}, \cite{YYZ25} and \cite{EL25} for related results.
\end{remark}

\medskip

{\bf Acknowledgement.} We thank Professors Shing-Tung Yau and Bong Lian for invitation to ICBS 2025 where this project was started. Fundamental works on singular K3 surfaces and Fermat hypersurfaces due to Late Professor Tetsuji Shioda inspired much for our work.

\section{Overlattice structure and proof of Theorem \ref{main2}} \label{sect1}

In this section, we first fix some notation used in this paper and then prove Theorem \ref{main2}. In what follows, we freely use basic facts on K3 surfaces explained in \cite{BHPV04}.

\medskip

A lattice $(L, (*,**))$ is a finitely generated free $\mathbb{Z}$-module $L$, endowed with a $\mathbb{Z}$-valued symmetric bilinear form $b_L=(*,**)$. For brevity, we often denote $(x,x)$ by $x^2$. We call $L$ an {\it even}  lattice if $x^2\in 2\mathbb{Z}$ for any $x\in L$.  When  $L$ is a sublattice of $L'$ of finite index, we call $L'$ an overlattice of $L$. Note that if $L'$ is an overlattice of a non-degenerate lattice $L$, then we have a canonical embedding $$L \subset L' \subset L^*:={\rm Hom}_\Z(L,\Z)\subset L\otimes_\Z\Q$$ by $(x, y) \in \Z$ for $x \in L$ and $y \in L'$. For a lattice $L = (L, b_L)$ and any non-zero integer $m$, we denote by $L(m)$ the lattice $(L, mb_L)$. 
 We denote by $U$ the unique even unimodilar hyperbolic lattice of rank $2$ and by $A_n$, $D_n$, $E_n$ the {\it negative definite} root lattice corresponding to the Dynkin diagram $A_n$, $D_n$, $E_n$ and by $\Lambda := U^{\oplus 3} \oplus E_8^{\oplus}$ the K3 lattice, which is nothing but the unique even unimodular lattice of signature $(3, 19)$ (See eg. \cite{Se12}). With respect to the cup product, we have $\Lambda \simeq H^2(V, \Z)$ for any K3 surface $V$.

\medskip

Let $V$ be a K3 surface. We denote by $S_V$ the N\'eron-Severi lattice and by $\rho(V) := {\rm rank}\, S_V$ the Picard number of $V$. $S_V$ is of signature $(1, \rho(V) -1)$. Let $T_V$ be the transcendental lattice of $V$, that is, the orthogonal complement of $S_V$ in $H^2(V, \Z).$ Then $T_V$ is of signature $(2, 20-\rho(V)).$ We set $(S_V)_{\Q} := S_V \otimes_{\Z} \Q$, $(S_V)_{\R} := S_V \otimes_{\Z} \R$ and similarly for $(T_V)_{\Q}$. Using the cup product, we naturally regard as
$$S_V \subset S_V^* = {\rm Hom}_{\Z}(S_V, \Z) \subset (S_V)_{\Q},$$
$$T_V \subset T_V^* = {\rm Hom}_{\Z}(T_V, \Z) \subset (T_V)_{\Q}.$$
Then $A_{S_V} := S_V^*/S_V$ and $A_{T_V} := T_V^*/T_V$ are finite abelian groups with quadratic forms $q_{A_{S_V}} : A_{S_V} \to \Q/2\Z$ and $q_{A_{T_V}} : A_{T_V} \to \Q/2\Z$ induced from the quadrtic form $q_V(x) = (x^2)$ on $H^2(V, \Z)$ such that
$$\iota : A_{S_V} \simeq H^2(V, \Z)/(S_V \oplus T_V) \simeq A_{T_V}\,\, ,\,\, q_{A_{S_V}}(v) + q_{A_{T_V}}(\iota(v)) = 0$$
in a natural manner. We call $(A_{S_V}, q_{A_{S_V}})$ the {\it discriminant group} of $V$. 

\medskip

We denote the group of isometries of $S_V$ by ${\rm O}\, (S_V)$ and naturally regard ${\rm O}\, (S_V) \subset {\rm O}\, ((S_V)_{\Q}) \subset {\rm O}\, ((S_V)_{\R})$. We denote by $W(S_V)$ the Weyl group of $S_V$, i.e., the $(-2)$-reflection group of $S_V$ generated by the $(-2)$-reflections $r_v$ on $S_V$:
$$r_v : x \mapsto x + (x, v)v,$$
where $v \in S_V$ with $(v^2) = -2$.
The positive cone of $V$, which we denote by $P_V$, is the connected component of
$$\{x \in (S_V)_{\R}\, |\, (x^2) := (x.x) > 0\}$$
containing the ample classes. Note that
$$W(S_V) \subset {\rm O}''(S_V) \subset {\rm O}'(S_V) \subset {\rm O}\, (S_V).$$
Let
$$\Aut(V, \omega_V) := {\rm Ker}\, (\Aut (V) \to {\rm O}\, (T_V)\, ;\, g \mapsto g^*|_{T_V}).$$
Then $\Aut(V,\omega_V)$ is nothing but the subgroup of $\Aut (V)$ acting on $H^0(V, \Omega_V^2) = \C \omega_V$ as identity. It is a finite index normal subgroup of $\Aut (V)$ such that $\Aut (V)/\Aut(V,\omega_V)$ is a finite cyclic group by the finiteness of canonical representation (\cite[Theorem 14.10]{Ue75}) together with the fact that any finite subgroup of $\C^{\times}$ is cyclic. By the global Torelli theorem for K3 surfaces \cite{PS71}, the map $\Aut (V,\omega_V) \to {\rm O}\, (S_V)$, $g \mapsto g^*|_{S_V}$, is injective. We regard $\Aut(V,\omega_V)$ as a subgroup ${\rm O}''(S_V)$
and we have 
$${\rm O}''(S_V) = W(S_V) \rtimes \Aut (V,\omega_V) .$$

\medskip

Fix a primitive embedding of $T_n$ into the K3 lattice $\Lambda_{K3}$ $$T_n= \left(\begin{array}{cc}
2^{n} & 0\\
0 & 2^{n}
\end{array}\right)\hookrightarrow \Lambda_{K3}:=U^{\oplus 3}\oplus E_8^{\oplus 2},$$ which is unique up to isometry (\cite[Theorem 1.14.4]{Ni80}).  We define $S_n:=(T_n)^{\perp}_{\Lambda_{K3}}$. Then $$S_n=U\oplus E_8^{\oplus 2}\oplus\left(\begin{array}{cc}
-2^n & 0\\
0 & -2^n
\end{array}\right).$$

For any integer $n \ge 2$, we define
$$V_{n} := {\rm Km}\, (E_{\sqrt{-1}} \times E_{2^{n-2} \sqrt{-1}}).$$
Then $\rho(V_n)=20$ and the surface $V_n$ ($n \ge 2$) is characterized, uniquely up to isomorphism, as the K3 surface with transcendental lattice $T_n$ up to the conjugate action of ${\rm SL}\, (2, \Z)$ by Shioda-Mitani \cite{SM74} and Shioda-Inose \cite{SI77}. Then $T_{V_n}\cong T_n$ and $S_{V_n}\cong S_n$ ($n\ge 2$). It is well known, e.g.
by Shioda \cite[Example 5.2]{Sh79}, that the lattice $T_3$ is also isometric to the transcendental lattice of $F_4^2$ up to the conjugate action of ${\rm SL}\, (2, \Z)$, thus $$F_4^2 \cong V_3.$$

It is also well known that $T_1$ is the transcendental lattice of the most algebraic K3 surface $X_4$ (\cite{Vi83}) and we denote $X_4=V_1$ while $X_4$ is not a Kummer surface. Then $T_{V_1}\cong T_1$ and $S_{V_1}\cong S_1$.

As primitive sublattices of $\Lambda_{K3}$, there is no inclusion relation between $T_n$ (resp. $S_n$) and $T_m$ (resp. $S_m$) for $m\neq n$. But as lattices, there exists an embedding from $T_n$ (resp. $S_n$) into $T_m$ (resp. $S_m$) for $m< n$.

\begin{lemma}\label{lem:isogeny}
For any integers $m$, $n$ with $1\le m < n$, the following statements hold:
\begin{enumerate}
\item There is a unique even overlattice $T$ of $T_n$ such that $T_n\subset T\subset T_n^*$ and $T\cong T_m$;
\item There is a unique even overlattice $S$ of $S_n$ such that $S_n\subset S\subset S_n^*$ and $S\cong S_m$;
\item Under the embedding $S_n\subset S_m$ in (2), any isometry of $S_n$ extends to $S_m$ i.e. $${\rm O}(S_n)\subset {\rm O}(S_m) \,\text{ and } \,{\rm O}'(S_n)\subset {\rm O}'(S_m).$$
\end{enumerate}
\end{lemma}

\begin{proof}
(1) and (2) follow from computation of discriminant forms and \cite[Proposition 1.4.1]{Ni80}. Then (3) follows from (2).
\end{proof}

For a finitely generated group $G$, we denote by $l(G)$ the minimal number of generators of $G$. For the trivial group $\{1\}$, we define $l(\{1\}) = 0$. Throughout this paper, we use this notation and the following well known fact on the group which is a direct consequence of \cite[Page 185, (6.14)]{Su82}.
\begin{proposition}\label{pp:Suzuki}
Let $N$ be a finitely generated group and $N_1<N$ with $s:=[N:N_1]<\infty$. Then
$l(N_1)\le (l(N)-1)s+1.$ In particular, $l(N_1)\le l(N)s$.
\end{proposition}

{\it Let us return back to the proof of Theorem \ref{main2}.}

\medskip

The following theorem follows from Vinberg \cite[Section 2, 2.2 and 2.4]{Vi83}:

\begin{theorem}\label{thm11} The group ${\rm O}''(S_1)$ is generated by $11$ elements.
\end{theorem}

\begin{proof} By \cite[Section 2, 2.2 and 2.4]{Vi83}, under the notation there, the group ${\rm O}'(S_{1})$ is generated by
$$2+1+1+ 2 = 6$$
elements from $\sS_2'$, $\sS_2''$, $\sS_1$ and generators of the symmetric group $S_5$ of degree $5$ respectively. Here we recall that the symmetric group $S_5$ is generated by $(12)$ and $(12345)$.

We have $A_{T_1} \simeq A_{S_1} = (\Z/2) v_1 \oplus (\Z/2) v_2$ with
$$q_{A_{S_1}}(v_1) = q_{A_{S_1}}(v_2) = 1/2\,\, ,\,\, q_{A_{S_1}}(v_1+v_2) =1\,\, {\rm  in}\,\, \Q/2\Z$$
for suitable $v_1, v_2 \in A_{S_1}$. From this, we find that ${\rm O}(A_{S_1}, q_{A_{S_1}}) \simeq \Z/2$. Then, by the definition of ${\rm O}''(S_1)$, we obtain
$$[{\rm O}'(S_1) : {\rm O}''(S_1)] \le |{\rm O}(A_{S_1}, q_{A_{S_1}})| =2.$$

Thus, by Proposition \ref{pp:Suzuki}, ${\rm O}''(S_1)$ is generated by $11$ elements. \end{proof}

 Let $v_n, u_n$ be a $\Z$-basis of $T_n(-1)$ whose Gram matix is
$$\left(\begin{array}{cc}
-2^{n} & 0\\
0 & -2^{n}
\end{array}\right).$$ 

Consider the embedding $i_n : T_{n+2} \to T_n$ in Lemma \ref{lem:isogeny} (1), i.e., 
$$i_n(v_{n+2}) = 2v_n\,\, ,\,\, i_n(u_{n+2}) = 2u_n.$$
Then, by Lemma \ref{lem:isogeny} (2),  $i_n$ gives the following isometry onto the image:
$$\iota_n := {\rm id}_{U \oplus E_8^{\oplus 2}} \oplus i_n : S_{n+2} \to S_n.$$
We regard $S_{n+2}$ as a sublattice of $S_{n}$, or equivalently, $S_{n}$ is an even overlattice of $S_{n+2}$ by $\iota_n$. By the definition of $\iota_n$, we have then $P(S_n) = P(S_{n+2})$ and
$$S_{n}/S_{n+2} = (\Z/2) [v_{n}] \oplus (\Z/2) [u_{n}]\simeq (\Z/2)^{\oplus 2},$$
where $[v_{n}]$ (resp. $[u_{n}]$) is the element of $S_{n}/S_{n+2}$ represented by $v_n$ (resp. $u_{n}$). By Lemma \ref{lem:isogeny}, we have $g(S_{n}) = S_n$ for any $g \in {\rm O}\,(S_{n+2})$.

\begin{proposition}\label{prop14} ${\rm O}''(S_{n+2})$ is a subgroup of ${\rm O}''(S_{n})$ under the canonical embedding $S_{n+2} \subset S_{n} \subset S_{n+2}^{*}$ given by $\iota_n$. Moreover, $[{\rm O}''(S_{n}) : {\rm O}''(S_{n+2})] \le 8 l$, where $l$ is the constant in Theorem \ref{main2}.
\end{proposition}

\begin{proof} Let $g \in {\rm O}''(S_{n+2})$. Since $g|_{S_{n+2}^{*}/S_{n+2}} = {\rm id}_{S_{n+2}^{*}/S_{n+2}}$, it follows that $g|_{S_{n}^*/S_{n+2}} = {\rm id}_{S_{n}^*/S_{n+2}}$ as well by $S_{n}^*/S_{n+2} \subset S_{n+2}^*/S_{n+2}$. Thus $g|_{S_{n}^*/S_{n}} = {\rm id}_{S_{n}^*/S_{n}}$ as well by $S_{n}^*/S_{n} = (S_{n}^*/S_{n+2})/(S_{n}/S_{n+2})$. Hence $g \in {\rm O}''(S_{n})$ and therefore ${\rm O}''(S_{n+2}) \subset {\rm O}''(S_{n})$.

\medskip

Let us show that $[{\rm O}''(S_{n}) : {\rm O}''(S_{n+2})] \le 8 l$.

\medskip

Note that ${\rm O}''(S_{n})$ acts equivariantly on $2S_{n} \subset S_{n}$ where $2S_{n} = \{2x\, |\, x \in S_{n}\}$. Then we obtain that
$$2S_{n} \subset S_{n+2} \subset S_{n},$$
as $S_{n}/S_{n+2} \simeq (\Z/2)^{\oplus 2}$, and that $S_{n}/2S_{n} \simeq (\Z/2)^{\oplus 20}$. Identify $\Z/2 = \F_2$. Then
$$\F_2^{\oplus 18} \simeq S_{n+2}/2S_{n} \subset S_{n}/2S_{n} \simeq \F_2^{\oplus 20}.$$
Therefore the orbit of $\{S_{n+2}\}$
$${\rm O}''(S_{n})\cdot \{S_{n+2}\} := \{g(S_{n+2})\, |\, g \in {\rm O}''(S_{n})\}$$
 is in one to one correspondence with some subset of the set ${\rm Gr}(18, 20)(\F_2)$ of $\F_2$ points of the Grassman variety ${\rm Gr}(18, 20)$.
Since
$$|{\rm Gr}(18, 20)(\F_2)| = |{\rm Gr}(2, 20)(\F_2)| = \frac{(2^{20} -1)(2^{20}-2)}{(2^2 -1)(2^2-2)} = l,$$
it follows that
$$|{\rm O}''(S_{n})\cdot \{S_{n+2}\}| \le l.$$
Let $B$ be the stabilizer of $\{S_{n+2}\}$ in ${\rm O}''(S_{n})$. As we already observed that ${\rm O}''(S_{n+2}) \subset {\rm O}''(S_{n})$, we have ${\rm O}''(S_{n+2}) \subset B={\rm O}''(S_{n})\cap {\rm O}'(S_{n+2})$. Hence
$$[{\rm O}''(S_{n}) : {\rm O}''(S_{n+2})] =k [{\rm O}''(S_{n}) : B] =k |{\rm O}''(S_{n})\cdot \{S_{n+2}\}|\le k l,$$
where $k:=[B:{\rm O}''(S_{n+2})]$. Thus, we are reduced to prove  $k\le 8$.

Note that each isometry $f\in B$ naturally induces an automorphism $\overline{f}\in {\rm O}(A_{S_{n+2}}, q_{A_{S_{n+2}}})$. Let $$H:=\{\overline{f}|\, f\in B\}\subset {\rm O}(A_{S_{n+2}}, q_{A_{S_{n+2}}}).$$ In order to prove $k\le 8$, it suffices to show that $|H|\le 8$.

Let $f\in B$. Since $f|_{A_{S_{n}}}={\rm id}_{A_{S_{n}}},$ we have $$f(\frac{u_n}{2^{n}})=\frac{u_n}{2^{n}}+\alpha \; \text{ and } \; f(\frac{v_n}{2^{n}})=\frac{v_n}{2^{n}}+\beta$$
for some $\alpha, \beta\in S_n$. Since $u_{n+2}=2u_n$ and $v_{n+2}=2v_n$, we have
$$f(\frac{u_{n+2}}{2^{n+2}})=\frac{u_{n+2}}{2^{n+2}}+\frac{\alpha}{2} \; \text{ and } \; f(\frac{v_{n+2}}{2^{n+2}})=\frac{v_{n+2}}{2^{n+2}}+\frac{\beta}{2}.$$ The elements $\frac{\alpha}{2}$ and $\frac{\beta}{2}$ are of order at most $4$ in $A_{S_{n+2}}$. Thus, we may rewrite
\begin{equation}\label{fbar}
\overline{f}(\frac{u_{n+2}}{2^{n+2}})=\frac{u_{n+2}}{2^{n+2}}+\frac{a_1 u_{n+2}}{4}+\frac{b_1 v_{n+2}}{4} \; \text{ and } \; \overline{f}(\frac{v_{n+2}}{2^{n+2}})=\frac{v_{n+2}}{2^{n+2}}+\frac{a_2 u_{n+2}}{4}+\frac{b_2 v_{n+2}}{4},
\end{equation}
where $a_i, b_i\in \{0,1,2,3\}$. For any pair of integers $(a, b)$, we have
\begin{equation}\label{qA}
q_{A_{S_{n+2}}}(a\frac{u_{n+2}}{2^{n+2}}+b\frac{v_{n+2}}{2^{n+2}})=\frac{a^2+b^2}{2^{n+2}}\;\; {\rm mod}\, 2\Z.
\end{equation}
Since $\overline{f}$ preserves the quadratic form $q_{A_{S_{n+2}}}$, by equations \eqref{fbar} and \eqref{qA}, we have
\begin{equation}
\frac{a^2+b^2}{2^{n+2}}=\frac{(a+2^{n}aa_1+2^{n}ba_2)^2+(2^{n}ab_1+b+2^{n}bb_2)^2}{2^{n+2}} \;\; {\rm mod}\, 2\Z.
\end{equation}
Then by direct computation via choosing $(a,b)=(1,0), (0,1), (1,1)$, we conclude that
\begin{enumerate}
\item If $n=1$, then the quadruple $(a_1,b_1,a_2,b_2)\in \{ (0, 0, 0, 0)$, $(0, 0, 0, 3)$, $(0, 2, 2, 0)$, $(0, 2, 2, 3)$, $(3, 0, 0, 0)$, $(3, 0, 0, 3), (3, 2, 2, 0), (3, 2, 2, 3)\}$;
\item If $n\ge 2$, then $(a_1,b_1,a_2,b_2)\in \{ (0,0,0,0), (0,1,3,0), (0,2,2,0), (0,3,1,0)\}$.
\end{enumerate}
This implies that $|H|\le 8$, which completes the proof of the proposition.
 \end{proof}

\begin{corollary}\label{cor11} Let $n\ge 1$ be an odd integer. Then ${\rm O}''(S_{n})$ is generated by $11\cdot (8l)^{\frac{n-1}{2}}$ elements, where $l$ is the constant in Theorem \ref{main2}. In particular, so is the quotient group $\Aut(V_n,\omega_{V_n}) = {\rm O}''(S_{n})/W(S_{n})$.
\end{corollary}

\begin{proof} We prove the first assertion by induction on odd integers $n\ge 1$. The result follows for $n=1$ by Theorem \ref{thm11}. Assume that $n \ge 1$ and the result follows for $n$. Then ${\rm O}''(S_{n})$ is generated by $11\cdot (8l)^{\frac{n-1}{2}}$ elements. Since $[{\rm O}''(S_{n}) : {\rm O}''(S_{n+2})] \le 8l$ by Proposition \ref{prop14}, it follows from Proposition \ref{pp:Suzuki} that ${\rm O}''(S_{n+2})$ is generated by $8l \cdot 11\cdot (8l)^{\frac{n-1}{2}} = 11\cdot(8l)^{\frac{n+1}{2}}$ elements. Hence the result is also true for $n+2$. This completes the proof.
\end{proof}

Now we are ready to prove Theorem \ref{main2}.

\medskip

{\it Proof of Theorem \ref{main2}.} Consider the element $f_4 \in \Aut (E_{\sqrt{-1}} \times E_{\tau})$ with
$\tau \in \H$ relevant to $V_n$ ($n\ge 2$) defined by
$$(x,y) \mapsto (\sqrt{-1}x,y).$$
Here $(x,y)$ is the global coordinate of the universal cover of
$E_{\sqrt{-1}} \times E_{\tau}$.

Then $f_4$ descends to an element $g_4$ of order $4$ of $\Aut (V_n)$ such that
$$g_4^*\omega_{V_n} = \sqrt{-1}\omega_{V_n}$$
 for each $n \ge 2$. Here $\omega_{V_n}$ is a nowhere vanishing holomorphic $2$-form on $V_n$ induced from $dx \wedge dy$. Thus the order of the action of $\Aut (V_n)$ on $T_n$ is divisible by $4$. On the other hand, since ${\rm rank}\, T_{n} = 2$, the order $m$ of the action of $\Aut (V_n)$ on $T_{n}$ is at most $6$, as the Euler function $\varphi(m) \le 2$. Hence the order of the action of $\Aut (V_n)$ on $T_{n}$ is exactly $4$. Hence
$$\Aut (V_n) = \Aut(V_n, \omega_{V_n}) \rtimes \langle g_4 \rangle.$$
Combining this with Corollary \ref{cor11}, for odd $n\ge 3$,  $\Aut (V_n)$ is generated by exactly $11\cdot (8l)^{\frac{n-1}{2}} + 1$ elements, as claimed. \qed

\medskip

\section{Fundamental domains of finite index subgroups} \label{sect2}

Throughout this section, we fix finitely generated groups $\tilde{H}$, $G$ and a subgroup $H$ of both $\tilde{H}$ and $G$ such that $[\tilde{H} : H] < \infty$ and $[G : H] < \infty$. Let $\R^{\rho}$ be the euclidean space with $\rho \ge 2$. Assume that $G$ acts faithfully on a {\it connected} cone $P \subset \R^{\rho}$ with vertx at the origin $0$ ($0$ may not be in $P$) such that there is an {\it exact tessellation} $\mathcal{T}_G:=\{D_i\}_{i\in I}$ of the action of $G$ on $P$ by $D_i \subset P$ ($i \in I$) of dimension $\rho$, which is a cone over a closed finite polyhedron in $P$ with vertex $0$. That is, $G$ acts on the set $\mathcal{T}_G$ transitively and each $D_i$ ($i \in I$) is an {\it exact fundamental domain} of the action of $G$ on $P$ in the following sense:

\begin{enumerate}

\item $$P = \bigcup_{g \in G} g(D_i),$$

\item $G_i := \{g \in G\,|\, g(D_i) = D_i\}$ is a finite group ({\it not necessarilly trivial in our definition}),

\item $D_i \cap D_{i'}$ ($i' \not= i$, $i' \in I$) is either empty or exactly one face of codimension $\ge 1$, possibly $\{0\}$, of both $D_i$ and $D_{i'}$,

\item For each codimension one face $\Delta$ of $D_i$, there is a unique $i' \in I$ such that $D_i \cap D_{i'} = \Delta$.

\end{enumerate}

We say that $D_{i'}$ is adjacent to $D_{i}$ ($i' \not= i$) when $D_i \cap D_{i'}$ is a codimension one face of both $D_i$ and $D_{i'}$. 

\medskip

{\it Throughout this section, we will work under the above setting.}

\medskip

Our goal is to obtain an upper bound for the minimal number of generators $l(\tilde{H})$ and $l(H)$ of the groups $\tilde{H}$ and $H$ (Theorems \ref{thm:bound}, \ref{thm:sharperbound}, \ref{thm:bound2}). We deduce our estimates by giving an effective way of finding a fundamental domain for the action of $H$ on $P$ via $\mathcal{T}_G$. More precisely, we want to find a subset $J\subset I$ such that
\begin{enumerate}
\item[(1a)] $$P =\bigcup_{g\in H,\,j\in J}g(D_j)\,\, (= \bigcup_{g\in H}g(\bigcup_{j \in J} D_j));$$
\item[(2a)] if $j,j'\in J$ and $j\neq j'$, then $g(D_j)\neq D_{j'}$ for any $g\in H$.
\end{enumerate}
A subset $J\subset I$ is called {\it $H$-fundamental} if (1a) and (2a) hold. From the properties of the tessallation $\mathcal{T}_G$, we have a transitive group action
$$\Psi_G:\, G\times \mathcal{T}_G\rightarrow \mathcal{T}_G, \,\, (g, D_i)\mapsto g(D_i).$$
By restriction, we get an action of $H$ on $\mathcal{T}_G$.

\begin{lemma}\label{thm:}
An $H$-fundamental subset exists and it must be a finite set.
\end{lemma}

\begin{proof}
Since $[G: H]<\infty$, there are only finitely many orbits for the action of $H$ on $\mathcal{T}_G$. The sets of representatives of the orbits for this action are in one to one correspondence with the $H$-fundamental subsets of $I$. \end{proof}

We choose and fix one $i_0 \in I$ and set:
$$K := G_{i_0} < G,$$
$$k := |\{ \text{codimension one faces of}\,\, G_{i_0}\}|.$$
Note that the conjugacy class of $K$ in $G$ and the number $k$ are independent of the choice of $i_0$ by the definition of tessellation.
Note also that $H$ and $K$ are subgroups of $G$ and for each $g\in G$, the $(H, K)$-double coset of $g$ is the set
$$H g  K:=\{h g k|\, h\in H,\, k\in K\}.$$
We denote by $H\backslash G/K$ the set of all $(H,K)$-double cosets. $H$-fundamental subsets are closely related to double cosets.
\begin{proposition}\label{prop:J}
Let $J\subset I$ be an $H$-fundamental subset. Then
$$|J|=|H\backslash G/K|.$$
\end{proposition}

\begin{proof}
Let $\mathcal{O}$ be the set of all orbits for the action of $H$ on $\mathcal{T}_G$. We define a map
$$\varphi: \mathcal{O}\rightarrow H\backslash G/K,\,\, [g(D_{i_0})]\mapsto HgK, \, \forall g\in G.$$
We need to check the well-definedness. If $[g(D_{i_0})]=[g'(D_{i_0})]$, then there exists $h\in H$ such that $h(g(D_{i_0}))=g'(D_{i_0})$. Thus, $g'^{-1}hg(D_{i_0})=D_{i_0}$ and $g'^{-1}hg\in K$. Then $g$ and $g'$ are in the same $(H,K)$-double coset and $\varphi$ is well-defined. In a similar way, one can check that $\varphi$ is injective. The surjectivity of $\varphi$ is clear.
\end{proof}

Note that $|H\backslash G/K|=|K\backslash G/H|$. Let $G/H$ be the set of all left cosets of $H$. Consider the action $K\times G/H\rightarrow G/H$, $(k, gH)\mapsto kgH$. Then the number of all orbits of this action is equal to $|K\backslash G/H|$.  Each $H$-fundamental subset $J$ of $I$  gives a (not necessarily connected) fundamental domain for the action of $H$ on $P$ in the sense of (1a)
and (2a). As remarked above, each $D_j$ $(j\in J)$ is bounded by $k$ codimension one faces. We now give a bound of the number of generators of $H$. 

In order to state our estimates, we introduce the following two subsets of $J \times I$:
$$\mathcal{A}_J:=\{(j,i) \in J \times I |\, D_j \text{ is adjacent to } D_i \},$$
$$\mathcal{N}_J:=\{(j,j') \in J \times J |\, D_j \text{ is adjacent to } D_{j'} \}.$$
\begin{theorem}\label{thm:bound}
Let $J\subset I$ be an $H$-fundamental subset. Then $H$ is finitely generated and
$$l(H)\le |{\mathcal A}_J| - |{\mathcal N}_J| + l(K\cap H).$$
In particular, setting $m = |J| = |H\backslash G/K|$ (cf. Proposition \ref{prop:J}) and recalling that $k$ is the cardinality of codimension one faces of any $D_i$ ($i \in I$), we have
$$l(H)\le mk+l(K\cap H).$$
\end{theorem}

\begin{proof}
Let $n:=l(K\cap H)$. Note that $|\mathcal{A}_J| = mk$ by the definition of $\mathcal{A}_J$. Then the second inequality follows from the first inequality. Let us show the first inequality.

By (1a), for each pair $(j,i)\in \mathcal{A}_J$, we may choose an element $g_{(j,i)}\in H$ such that $g_{(j,i)}(D_{j'})=D_i$ for some $j'\in J$. Note that if $(j, j') \in \mathcal{N}_J$, then, as we have ${\rm id}_{H}(D_{j'}) = D_{j'}$ and $j' \in J$, we may choose $g_{(j,j')} = {\rm id}_{H}$ and thus we can exclude them from any set of generators of any subgroup of $H$. By (1a) and (2a), there exists a unique $j_0\in J$ such that $g_0(D_{j_0})=D_{i_0}$ for some $g_0\in H$. Let $B_{j_0}$ be a set of generators of $G_{j_0}\cap H$ with $|B_{j_0}|=l(G_{j_0}\cap H)$. We have $|B_{j_0}| = n$ as $g_0(G_{j_0} \cap H) g_0^{-1} = G_{i_0} \cap H$ by the choice of $g_0$. We define
$$B:= \{g_{(j,i)}|\, (j,i)\in \mathcal{A}_J\}\cup B_{j_0}\,\, ,\,\, \Gamma := \langle B \rangle.$$
Here, in the first subset of $B$, the cardinality of the set of indices $(j,i)$ in the first subset of $B$ is exactly $|{\mathcal A}_J|$ and among them we have at least  $|{\mathcal N}|$ elements such that $g_{(j,i)} = {\rm id}_H$ corresponding to the subset $|{\mathcal N}|$ as remarked above.  So, it suffices to show that the subgroup $\Gamma$ generated by $B$ is equal to $H$.

Consider the set
$$I':=\{i\in I|\, D_i=g(D_j) \text{ for some }g\in \Gamma\text{ and some }j\in J\} (\subset I).$$
Obviously, $J\subset I'$. We claim that

\medskip

{\it if $i'\in I'$ and $D_{i}$ is adjacent to $D_{i'}$ across a hyperplane bounding $D_{i}$ and $D_{i'}$, then $i\in I'$.}

\medskip

In fact, by $i'\in I'$, there exist $g\in \Gamma$ and $j\in J$ such that $g(D_j)=D_{i'}$. Then $g^{-1}(D_i)$ is adjacent to $g^{-1}(D_{i'})=D_j$. Note that $g^{-1}(D_i)=D_{i''}$ for some $i''\in I$ and $(j,i'')\in \mathcal{A}_J$. Thus, $g_{(j,i'')}(D_{j'})=D_{i''}$ for some $j'\in J$. Then $g g_{(j,i'')}(D_{j'})=D_i$ and $i\in I'$. This proves the claim.

By the claim and by the {\it connectivity} of $P$ assumed at the beginning, we infer that $I'=I$.

Let $h\in H$. By $I'=I$,  we have $h(D_{j_0})=g_1(D_{j_1})$ for some $j_1\in J$ and some $g_1\in\Gamma$. Then $g_1^{-1}h(D_{j_0})=D_{j_1}$. By (2a), $j_0=j_1$ and $g_1^{-1}h\in G_{j_0}\cap H$. Thus, $h\in \Gamma$ by the definition of $B$ and $\Gamma$. This completes the proof. \end{proof}

\begin{remark}
The proof provides a method to obtain an explicit set $B$ of generators of $H$. The bound in the theorem is not optimal in general.
\end{remark}

We also have the following refined bound of $l(H)$, which turns out to be useful in our application in the next section:

\begin{theorem}\label{thm:sharperbound}
Let $J$ be an $H$-fundamental subset with $m = |J| = |H\backslash G/K|$ as in the previous theorem. Let
$$J := \{j_1, j_2, \ldots, j_m\} \subset I,$$
$$H_{j_t} := G_{j_t} \cap H$$
where $G_{j_t} = \{g \in G\, |\, g(D_{j_t}) = D_{j_t}\}$ ($1 \le t \le m$) is the stabilizer of $D_{j_t}$. Then $H_{j_t}$ acts on the (finite) set of codimension one faces of $D_{j_t}$. We denote by $a_t$ the cardinality of the orbits of this action. Then
$$l(H)\le \sum_{t=1}^{m} (l(H_{j_t}) + a_t) + l(K\cap H)  - |{\mathcal N}_J|.$$
In particular,
$$l(H)\le \sum_{t=1}^{m} (l(H_{j_t}) + a_t) + l(K\cap H).$$
\end{theorem}

\begin{proof} For each $1 \le t \le m$, we set $b_t = l(H_{j_t})$ and choose generators of $H_{j_t}$ with cardinality $b_t$ and denote them by
$$h_{j_t, 1}, h_{j_t, 2}, \ldots, h_{j_t, b_t}.$$
Let us write the orbits of the action of $H_{j_t}$ on the set of codimension one faces of $D_{j_t}$ by
$$O_{j_t, 1}, O_{j_t, 2}, \ldots, O_{j_t, a_t}$$
and choose a set of complete representatives as
$$\Delta_{j_t, 1}, \Delta_{j_t, 2}, \ldots, \Delta_{j_t, a_t},$$
where $\Delta_{j_t, c} \in O_{j_t, c}$.
For each $1 \le c \le a_t$, we denote by $D_{i_{t, c}}$ the unique element of $T_G =\{D_i\}_{i \in I}$ such that
$$D_{i_{t,c}} \cap D_{j_t} = \Delta_{j_t, c}\,\, (i_{t, c} \in I)$$
Here, the existence and uniqueness of $D_{i_{t, c}}$ follows from the definition of tessellation. By the definition of $\mathcal{A}_J$, we have $(j_t, i_{t,c}) \in \mathcal{A}_J$.

By (1a), for each pair $(j_t, i_{t,c}) \in \mathcal{A}_J$, we may choose an element $g_{(j_t, i_{t,c})}\in H$ such that $g_{(j_t, i_{t,c})}(D_{j'})=D_{i_{t, c}}$ for some $j'\in J$. As before, if $i_{t, c} \in J$, we may and will choose $j' = i_{t, c}$ and $g_{i_t,c}= {\rm id}_{H}$ and thus we can exclude them from any set of generators of any subgroup of $H$. By slightly abuse of notation, we choose $j_0 \in J$ (which is indeed $j_t$ for some $t$ with $1 \le t \le m$) such that $h(D_{j_0}) = D_{i_0}$ for some $h \in H$ as before and let $B_{j_0}$ be a set of generators of $G_{j_0}\cap H$ with $|B_{j_0}|=l(G_{j_0}\cap H)$. We define
$$B' := \bigcup_{1 \le t \le m, 1 \le c \le a_t} \{g_{(j_t,i_{t, c})}\}\,\, \bigcup \bigcup_{1 \le t \le m, 1 \le d \le b_t} \{h_{j_{t}, d}\}\,\, \bigcup B_{j_0}\,\, (\subset H)\,\, ,\,\, \Gamma' := \langle B' \rangle \subset H.$$

Then, exactly for the same reason as in the proof of the previous theorem, it suffices to show that $\Gamma' = H$. For this, under the description of $H$ in the previous theorem, it suffices to show the following claim:

\medskip

{\it We may choose $g_{(j, i)} \in \Gamma'$ for all $(j, i) \in {\mathcal A}_J$.}

\medskip

We now show the claim. Set $j = j_t$ ($1 \le t \le m$).

By the definition of $g_{(j_t, i)}$, there is $j' \in J$ such that
$$g_{(j_t, i)}(D_{j'}) = D_i.$$
Let $\Delta$ be the unique codimension one face of both $D_{j_t}$ and $D_i$ such that $\Delta = D_{j_t} \cap D_i$. Then, there are $h \in H_{j_t}$ and $1 \le c \le a_t$ such that $\Delta = h(\Delta_{j_t, c})$ by the definition of the set $\{\Delta_{j_t, c}\}_{c=1}^{a_t}$. Since $h(D_{j_t}) = D_{j_t}$, by the uniqueness of $D_i$, we obtain $h(D_{i_{t, c}}) = D_i$.  On the other hand,
$$g_{(j_t, i_{t, c})}(D_{j''}) = D_{i_{t, c}}$$
for some $j'' \in J$ by $(j_t, i_{t, c}) \in {\mathcal A}_J$. Applying $h$ on the both sides and using $h(D_{i_{t, c}}) = D_i$, we obtain
$$hg_{(j_t, i_{t, c})}(D_{j''}) = h(D_{i_{t, c}}) = D_i.$$
For any choice of $g_{(j_t, i)}$, there is $j' \in J$ such that $g_{(j_t, i)}(D_{j'}) = D_i$ with $j' \in J$. Then $D_{j'} = D_{j''}$ by $j', j'' \in J$ and by the condition (2a) of an $H$-fundamental set $J$. Then we may just choose
$$g_{(j_t, i)} = hg_{(j_t. i_{t, c})} \in \Gamma'.$$
This completes the proof. \end{proof}

From now on, we estimate $l(\tilde{H})$ from the above.

\begin{theorem}\label{thm:bound2} Under the setting at the beginning of this section, we have:
$$l(\tilde{H}) \le l(H) + [\tilde{H} : H] -1.$$
In particular, under the notation in Theorem \ref{thm:bound}, we have an effective estimate
$$l(\tilde{H}) \le mk+ l(K\cap H) + [\tilde{H} : H] -1.$$
\end{theorem}

\begin{remark} We obtain a more refined estimate of $l(\tilde{H})$ by substituting the estimate $l(H)$ in Theorem \ref{thm:sharperbound}.
\end{remark}

\begin{proof} Since $\tilde{H}$ is generated by the generators of $H$ and the complete representatives of $\tilde{H}/H$ with $1_{\tilde{H}}$ removed, the result follows.
\end{proof}

\begin{corollary}\label{cor:boundgeneral} Let $\tilde{L}$ be a group. Under the setting at the beginning of this section, we assume that there is a normal subgroup $\tilde{W}$ of $\tilde{H}$ and the group homomorphism
$$p : \tilde{L} \to \tilde{H}/\tilde{W}$$
such that $|{\rm Ker}\, (p)| < \infty$ and $s := [(\tilde{H}/\tilde{W}) : {\rm Im}\, (p)] < \infty$. Then
$$l(\tilde{L}) \le l({\rm Ker}\, (p)) + s(l(\tilde{H}/\tilde{W})-1) + 1.$$ In particular,
$$l(\tilde{L}) \le l({\rm Ker}\, (p)) + s(l(\tilde{H})-1) + 1.$$
\end{corollary}

\begin{proof} Let $K = {\rm Ker}\, (p)$ and $M := {\rm Im}\, (p)$.
Then $s = [\tilde{H}/\tilde{W} : M]$ and thus
$$l(M) \le s(l(\tilde{H}/\tilde{W})-1) + 1$$
by Proposition \ref{pp:Suzuki}. By the exact equence of the groups
$$1 \to K \to \tilde{L} \to M \to 1$$
it follows that $l(\tilde{L}) \le l(K) + l(M)$. This implies the first inequality. The second inequality is then clear as $l(\tilde{H}/\tilde{W}) \le l(\tilde{H})$.
\end{proof}

\section{Proof of Theorem \ref{main3}} \label{sect3}

In this section, we apply our bounds in the previous section to prove Theorem \ref{main3}.

We use the same notation of lattices as in Section \ref{sect1}. For instance, recall that $S_1$, $S_2$ and $S_3$ are the N\'{e}ron-Severi lattices of the K3 surfaces $V_1=X_4$, $V_2={\rm Km}\, (E_{\sqrt{-1}} \times E_{\sqrt{-1}})$ and $V_3=F_4^2$ respectively. The lattice $S_1$ has a $\Z$-basis $e_1,e_2,\dots,e_{18},\alpha,\beta$ such that $U\oplus E_8^{\oplus 2}=\langle e_1,e_2,\dots,e_{18}\rangle$ and $\alpha^2=\beta^2=-2$, $(\alpha,\beta)=0$, $(\alpha,e_i)=(\beta,e_i)=0$ for all $i$. Then
$$S_3=\langle e_1,e_2,\dots,e_{18},2 \alpha,2\beta\rangle \subset S_2=\langle e_1,e_2,\dots,e_{18},\alpha+\beta,\alpha-\beta\rangle\subset S_1.$$

\begin{lemma}\label{lem:f}
Let $f\in {\rm O}'(S_1)$. We write
$$f(\alpha)=\sum_{i=1}^{18} n_i e_i+a \alpha + b \beta.$$
Then
\begin{enumerate}
\item all $n_i$ are even;
\item exactly one of $a$ and $b$ is even.
\end{enumerate}
\end{lemma}

\begin{proof}
Since $(\alpha, x)$ is even for any $x\in S_1$, we have $(f(\alpha),x)$ is even. Since $U\oplus E_8^{\oplus 2}$ is unimodular, for any $1\le i\le 18$, there exists $y_i\in \langle e_1,\dots,e_{18}\rangle$ such that $(y_i,e_i)=1$ and $(y_i,e_j)=0$ ($i\neq j$). Then $n_i=(f(\alpha),y_i)$ is even and (1) holds. By $f(\alpha)$ being primitive in $S_1$ and $-2=\alpha^2=f(\alpha)^2$, we infer that exactly one of $a$ and $b$ is even and (2) holds.
\end{proof}

Let $f\in {\rm O}'(S_2)\subset {\rm O}'(S_1)$. By Lemma \ref{lem:f}, we may write
$f(\alpha)=\sum_{i=1}^{18} 2n_i e_i+a \alpha + b \beta$. We call $$v_F(f):=(\overline{n_1},\overline{n_2},\dots,\overline{n_{18}})\in \mathbb{F}_2^{18}$$ the {\it associated vector} of $f$.

\begin{lemma}\label{lem:vF}
Let $f,g\in {\rm O}'(S_2)\subset {\rm O}'(S_1)$. Then the following statements are equivalent:
\begin{enumerate}
\item $f(S_3)=g(S_3)$;
\item $f(\frac{\alpha}{4})\equiv g(\frac{\alpha}{4})$ mod $S_2^*$;
\item $v_F(f)=v_F(g)$.
\end{enumerate}
\end{lemma}
\begin{proof}
We write
$f(\alpha)=\sum_{i=1}^{18} 2n_i e_i+a \alpha + b \beta$
and
$g(\alpha)=\sum_{i=1}^{18} 2n_i' e_i+a' \alpha + b' \beta.$
Note that $f(S_3)=g(S_3)$ if and only if $f(S_3^*)=g(S_3^*)$. Since $S_3^*=\langle \frac{\alpha}{4},S_2^*\rangle$ and $f(S_2^*)=g(S_2^*)=S_2^*$, we have $f(S_3)=g(S_3)$ if and only if $f(\frac{\alpha}{4})\equiv g(\frac{\alpha}{4})$ mod $S_2^*$, which implies the equivalence between (1) and (2).

By $f(\frac{\alpha}{4})- g(\frac{\alpha}{4})=\sum_{i=1}^{18} \frac{n_i-n_i'}{2} e_i+\frac{a-a'}{4} \alpha + \frac{b-b'}{4} \beta$, $S_2^*=\langle e_1,\dots,e_{18}, \frac{\alpha+\beta}{4},\frac{\alpha}{2}\rangle$ and Lemma \ref{lem:f}, we infer the equivalence between $(2)$ and $(3)$.
\end{proof}

\begin{corollary}\label{cor:index}
The index $[{\rm O}'(S_2):{\rm O}'(S_3)]\leq 2^{18}$.
\end{corollary}

\begin{proof}
We define $\varphi: {\rm O}'(S_2)/{\rm O}'(S_3)\rightarrow \mathbb{F}_2^{18}$, $g {\rm O}'(S_3)\mapsto v_F(g)$. By Lemma \ref{lem:vF}, $\varphi$ is a well-defined injective map. Thus, $|{\rm O}'(S_2)/{\rm O}'(S_3)|\leq 2^{18}$. \end{proof}

We are now ready to prove Theorem \ref{main3}.

\medskip

{\it Proof of Theorem \ref{main3}.}
Recall that  $V_2={\rm Km}\, (E_{\sqrt{-1}} \times E_{\sqrt{-1}})$ and $S_{2}={\rm NS}(V_2)$ is an even overlattice of $S_{3}={\rm NS}(F_4^2)$ with ${\rm O}'(S_{3})\subset {\rm O}'(S_{2})$ (Lemma \ref{lem:isogeny}). Let $G:={\rm O}'(S_2)$, $\tilde{H}=H:={\rm O}'(S_3)$, and let $P$ be the positive cone $P_{V_2}\subset S_{2}\otimes \R$ of $V_2$.

By \cite[Lemmas 3.2, 3.5, 3.6]{KK01}, $P$ has a tessellation $\mathcal{T}_{G}=\{D_i\}_{i\in I}$ preserved by ${\rm O}'(S_2)$, each fundamental domain $D_i$ is bounded by exactly $k=624=40+40+64+160+320$ hyperplanes, and the symmetric group $G_i$ of $D_i$ is of order $7680$.
As before, we let  $K=G_{i_0}$ for some fixed $i_0\in I$. By computation with help of PC, we find a subset of $G/H$ contains $2^{18}$ different left cosets of $H$ in  $G$. Then by Corollary \ref{cor:index},  $G/H$ has exactly $2^{18}$ left cosets. Note that $|H\backslash G/K|=|K\backslash G/H|$. Consider the action $$K\times G/H\rightarrow G/H, \,\,(k, gH)\mapsto kgH.$$ Then the number of all orbits of this action is equal to $|K\backslash G/H|=|J|$ for any $H$-fundamental subset $J$ of $I$ by Proposition \ref{prop:J}. By computation with help of PC,  the action of $K$ on $G/H$ has exactly $112$ orbits and we find an explicit $H$-fundamental subset $J=\{j_1,...,j_{112}\}$ of $I$. Moreover, the finite groups $H_{j_t}=G_{j_t}\cap H$ ($1\le t\le m=112$) and the number $a_t$ of orbits of the actions of $H_{j_t}$ on codimension one faces of $D_{j_t}$ are summarized in Table \ref{tab:orbits}. Moreover, for our choice of $J$, the set $\mathcal{N}_J$ has cardinality $232$ by computation with help of PC. Thus, by Theorem \ref{thm:sharperbound} and Table \ref{tab:orbits}, we get
$$l(H)\le 24246+0-232=24014.$$
By removing repeated ones among these $24014$ isometries using computer, we obtain an explicit set $\mathcal{S}$ of generators of ${\rm O}'(S_3)$ with the cardinality $|\mathcal{S}|=19361$. An explicit description of $D_{j_t}$, the generators of $H_{j_t}$, the $H_{j_t}$-orbits, $\mathcal{N}_J$, and $\mathcal{S}$ can be found in the ancillary file fundamentaldomains.txt.

Let $\tilde{L}$ and $\tilde{W}$ be ${\rm Aut}(F_4^2)$ and the Weyl group $W(S_F)$ of $S_F$ respectively. Consider the natural homomorphism
$$p: {\rm Aut}(F_4^2)=\tilde{L}\rightarrow \tilde{H}/\tilde{W}={\rm O}'(S_F)/W(S_F).$$
Then $|{\rm ker}(p)|=1$ and $s=[\tilde{H}/\tilde{W}:{\rm Im}(p)]=4$ by Torelli theorem of K3 surfaces (\cite{PS71}) and by the fact that the action of ${\rm Aut}(F_4^2)$ on the holomorphic $2$-forms being of order $4$ (see the proof of Theorem \ref{main2}). Thus, by Corollary \ref{cor:boundgeneral}, we have
$$l({\rm Aut}(F_4^2))=l(\tilde{L})\le 0+4(19361-1)+1=77441.$$
This completes the proof of Theorem \ref{main3}. \qed

\begin{table}[htbp]
\centering
\caption{The number $a_t$ of orbits of $H_{j_t}$ acting on the codimension $1$ faces of $D_{j_t}$}\label{tab:orbits}
{\footnotesize
\begin{tabular}{{|ccc|ccc|ccc|ccc|ccc|}}
\hline
$H_{j_t}$ & $l(H_{j_t})$ & $a_t$ &$H_{j_t}$ & $l(H_{j_t})$ & $a_t$  & $H_{j_t}$ & $l(H_{j_t})$ & $a_t$ & $H_{j_t}$ & $l(H_{j_t})$ & $a_t$ & $H_{j_t}$ & $l(H_{j_t})$ & $a_t$  \\
\hline
$1$& 0 & 624 & $C_2$ & 1 & 342 & $C_2^2$ & 2 & 189 & $C_2^2$ & 2 & 189 & $D_{12}$ & 2 &75  \\
$1$ & 0 & 624 & $C_2$ & 1 & 342 & $C_2^2$& 2 & 198 & $C_2^2$ & 2 & 189 & $D_{12}$ & 2 & 75 \\
$1$ & 0 & 624 & $C_2$ & 1 & 342 & $C_2^2$ & 2 & 189 & $C_2^2$ & 2 & 189 & $C_2\times D_{8}$ & 3 & 62 \\
$1$ & 0 & 624 & $C_2$ & 1 & 318 & $C_2^2$ & 2 & 175 & $C_2^3$ & 3 & 103 & $C_2\times D_{8}$ & 3 & 71 \\
$1$& 0 & 624 & $C_2$ & 1 & 342 & $C_2^2$ & 2 & 175 & $C_2^3$ & 3 & 117 & $C_2\times D_{8}$ & 3 & 62 \\
$1$& 0 & 624 & $C_2$ & 1 & 342 &$C_2^2$ & 2 & 175 & $C_2^3$ & 3 & 103 & $C_2\times D_{8}$ & 3 & 71 \\
$1$& 0 & 624 & $C_2$& 1 & 314 & $C_2^2$ & 2 & 175 & $C_2^3$ & 3 & 103 & $C_2\times D_{8}$ & 3 & 71 \\
$1$& 0 & 624 & $C_3$ & 1 & 222 & $C_2^2$ & 2 & 175 & $C_2^3$ & 3 & 104 & $C_2\times D_{8}$ & 3 & 62 \\
$1$& 0 & 624 & $C_3$ & 1 & 222 & $C_2^2$ & 2 & 198 & $C_2^3$ & 3 & 120 & $C_2\times D_{8}$ & 3 & 62 \\
$1$& 0 & 624 & $C_2^2$ & 2 & 175 & $C_2^2$ & 2 & 175 & $C_2^3$ & 3 & 117 & $C_2\times D_{8}$ & 3 & 62 \\
$C_2$ & 1 & 342 & $C_2^2$ & 2 & 163 & $C_2^2$ & 2 & 175 & $C_2^3$ & 3 & 120 & $C_2\times D_{8}$ & 3 & 62 \\
$C_2$ & 1 & 342 & $C_2^2$ & 2 & 198 & $C_2^2$ & 2 & 189 & $C_2^3$ & 3 & 104 & $C_2\times D_{8}$ & 3 & 71 \\
$C_2$ & 1 & 314 & $C_2^2$ & 2 & 189 & $C_2^2$ & 2 & 189 & $C_2^3$ & 3 & 103 & $C_2\times D_{8}$ & 3 & 71 \\
$C_2$ & 1 & 318 & $C_2^2$ & 2 & 175 & $C_2^2$ & 2 & 175 & $C_2^3$ & 3 & 104 & $C_2\times D_{8}$ & 3 & 71 \\
$C_2$ & 1 & 342 & $C_2^2$ & 2 & 198 & $C_2^2$ & 2 & 189 & $C_2^3$& 3 & 120 & $C_2^2\times D_{10}$ & 3 & 30 \\
$C_2$ & 1 & 342 & $C_2^2$ & 2 & 198 & $C_2^2$ & 2 & 189 & $C_2^3$ & 3 & 104 & $C_2^2\times D_{10}$ & 3 & 30 \\
$C_2$ & 1 & 342 & $C_2^2$ & 2 & 189 & $C_2^2$ & 2 & 163 & $C_2^3$ & 3 & 120 & $C_2\times \mathfrak{S}_4$ & 2 & 26 \\
$C_2$ & 1 & 342 & $C_2^2$ & 2 & 198 & $C_2^2$ & 2 & 177 & $C_2^3$ & 3 & 104 & $C_2\times \mathfrak{S}_4$ & 2 & 35 \\
$C_2$ & 1 & 342 & $C_2^2$ & 2 & 175 & $C_2^2$ & 2 & 177 & $C_2^3$ & 3 & 104 & $C_2\times \mathfrak{S}_4$ & 2 & 26 \\
$C_2$ & 1 & 318 & $C_2^2$ & 2 & 198 & $C_2^2$ & 2 & 198 & $C_2^3$ & 3 & 117 & $C_2\times \mathfrak{S}_4$ & 2 & 35 \\
$C_2$ & 1 & 342 &$C_2^2$ & 2 & 175 & $C_2^2$ & 2 & 198 & $C_2^3$ & 3 & 117 &  &  &  \\
$C_2$ & 1 & 342 & $C_2^2$ & 2 & 198 & $C_2^2$ & 2 & 198 & $D_{12}$ & 2 & 61 &  &  &  \\
$C_2$ & 1 & 318 & $C_2^2$ & 2 & 189 & $C_2^2$ & 2 & 198 & $D_{12}$ & 2 & 61 &  &  &  \\
\hline
\end{tabular}
In this table, $C_n$ is the cyclic group of order $n$ and $D_n$ is the dihedral group of order $n$.}

\end{table}

\section{Dirichlet reduction and Proof of Theorem \ref{main1}}\label{sect5}
In order to get a generating set of ${\rm Aut}(F_4^2)$ with small cardinality, we introduce a new method of reducing the number of generators considering the Dirichlet domain. Let $V$ be a K3 surface with Picard number $\rho(V)\ge 3$. Let $S$ denote the N\'eron-Severi lattice of $V$ and $P\subset S\otimes \R$ the positive cone containing an ample class $h$ of $V$. We call 
$$D_h:=\{x\in P|\, (x, h)\le (x, g(h))\text{ for all } g\in {\rm O}'(S)\}$$
{\it the Dirichlet domain} with respect to $h$. We do not require the stabilizer of $h$ in ${\rm O}'(S)$ is trivial.
\begin{proposition}\label{pp:Dirichlet}
The Dirichlet domain $D_h$ is a finite rational polyhedral cone. In particular, there exist a positive integer $k$ and $g_1,g_2,...,g_k\in {\rm O}'(S)$ such that 
$$\Delta_i:=\{x\in P|\, (x,g_i(h)-h)=0\},\,\, 1\le i\le k,$$ 
are $k$ different hyperplanes bounding $D_h$ and
$$D_h=\{x\in P|\, (x, h)\le (x, g_i(h))\text{ for }1\le i\le k \}.$$ Moreover, for any $f\in {\rm O}'(S)$, 
$$\text{either } f(h)=h \text{ or } (f(h), g_i(h)-h)<0 \text{ for some }i.$$ In particular, ${\rm O}'(L)$ is generated by $g_1,g_2,...,g_k$ and $\{g\in {\rm O}'(S)|\, g(h)=h\}.$
\end{proposition}

\begin{proof}
The existence of a finite set of isometries $g_i$ is a consequence of the proof of the cone conjecture for K3 surface by Sterk (\cite[Page 511]{St85}). Indeed, $D_h$ is a finite rational polyhedral cone by \cite[Lemma 2.3]{St85}. Let $\Delta_i$ ($1\le i\le k$) be the bounding hyperplanes of $D_h$. Then there exist  $g_i\in {\rm O}'(S)$ ($1\le i\le k$) such that $g_i$ define the hyperplanes $\Delta_i$ which bound $D_h$.  

For an isometry $f\in {\rm O}'(S)$, if $f(h)\neq h$, then $$(f(h),f(h)-h)=(f(h),f(h))-(f(h),h)=(h,h)-(f(h),h)<0$$ by Hodge index theorem. Indeed, 
$$(f(h), h)^2> (h,h)(f(h),f(h))=(h,h)^2 $$
as $f(h)$ and $h$ are linearly independent by $f(h)\neq h$. Thus, if $f(h)\neq h$, then $f(h)\notin D_h$ by the definition of $D_h$ and $(f(h),g_i(h)-h)<0$ for some $i$. 

We claim that $g_1,g_2,...,g_k$ and the finite group $\{g\in {\rm O}'(S)|\, g(h)=h\}$ generate ${\rm O}'(S)$. If $f(h)\neq h$, then  $(f(h), g_i(h)-h)<0$, or equivalently, $(g_{i}^{-1}f(h),h)<(f(h),h)$ for some $i$. If $g_{i}^{-1}f(h)\neq h$, then $(g_{i}^{-1}f(h),g_j(h)-h)<0$ for some $j$. Hence $$(g_j^{-1}g_i^{-1}f(h),h)<(g_i^{-1}f(h),h)<(f(h),h).$$
Repeating this process, we get $1\le i_1,...,i_l\le k$ such that $g_{i_l}^{-1}\cdots g_{i_1}^{-1}f(h)=h$, which implies the claim. \end{proof}

In general it is difficult to compute such $g_i$ and the bounding hyperplanes $\Delta_i$. Proposition \ref{pp:Dirichlet} and its proof motivate the following new method of reducing the number of generators.

\medskip

\noindent
{\bf Dirichlet reduction.} Let $\mathcal{S}=(f_1,f_2,...,f_m)$ be a finite sequence of elements in ${\rm O}'(S)$. We call $$d({\mathcal{S},h}):=\Sigma_{i=1}^{m} (f_i(h),h)$$ {\it the total $h$-degree} of $\mathcal{S}$. We use a word $w:=(a_1,...,a_l)$ ($a_i\in \{\pm 1, \pm 2,...,\pm m\}$) of length $l$ to represent an isometry $f_w:=f_{a_1}f_{a_2}\cdots f_{a_l}$. Here we adopt the convention that $f_{-i}:=f^{-1}_i$ for $1\le i \le m$. If $i$ appears exactly once in the $l$-tuple $(|a_1|,|a_2|,...,|a_l|)$ and $$(f_w(h),h)<(f_i(h),h),$$ then we say $w$ {\it replaces $i$}. If $w$ replaces $i$, then $$\mathcal{S}_{(w,i)}:=(f_1,f_2,...,f_{i-1},f_w,f_{i+1},...,f_m)$$ and $\mathcal{S}$ generate the same subgroup of ${\rm O}'(S)$ while $$d(\mathcal{S}_{(w,i)},h)<d(\mathcal{S},h).$$ We search words $w$ of reasonable length and if we find a word $w$ replacing some $i$, then we replace $\mathcal{S}$ by $\mathcal{S}_{(w,i)}$ (if $f_w$ is equal to $f_j$ for some $j\neq i$, then we replace $\mathcal{S}$ by $\mathcal{S}\setminus \{f_i\}$). We repeat this process until either $\mathcal{S}$ becomes reasonably small or it is difficult to find a new word repacing some $i$ within reasonable time.

\begin{remark}\label{rmk:Dirichlet} (1) For an integer $b>0$, $G_b:=\{f\in {\rm O}'(S)|\, (f(h),h)\le b\}$ is a finite set. In fact, the action of ${\rm O}'(S)$ on $P$ is discrete with finite stabilizer of $h$, which implies the finiteness of $G_b$. Its cardinality $|G_b|$ decreases when $b$ decreases. If $b=h^2$, $G_b$ is the stabilizer of $h$ in ${\rm O}'(S)$. This is the reason why Dirichlet reduction is useful to reduce the number of generators. It turns out that this method works quite effectively in our study of ${\rm Aut}(F_4^2)$ as we will explain below.

(2) Suppose $f$ is in the group generated by $\mathcal{S}=(f_1,f_2,...,f_m)$. The Dirichlet reduction may be used to find an explicit word for $f$. More precisely, we replace $\mathcal{S}$ by $(f_1,f_2,...,f_m, f)$ and proceed Dirichlet reduction by only searching for words repacing $m+1$. 
\end{remark}

Let $\xi=e^{\frac{2\pi i}{8}}$. The Fermat quartic K3 surface $F_4^2$ contains the following $48$ lines: $$L_{ij}:=\{x_0-\xi^i x_1=x_2-\xi^j x_3=0\},$$ $$M_{ij}:=\{x_0-\xi^i x_2=x_1-\xi^j x_3=0\},$$ $$N_{ij}:=\{x_0-\xi^i x_3=x_1-\xi^j x_2=0\},$$ where $i,j\in \{1,3,5,7\}$. Let $h\in {\rm NS}(F_4^2)$ denote the hyperplane class. Note that $R_1:=h-L_{77}-N_{13}$, $R_2:=h-M_{57}-N_{13}$, $R_3:=h-M_{33}-N_{33}$, $R_4:=h-L_{11}-M_{33}$ are the classes of four smooth conics on $F_4^2$.  Consider elliptic fibrations and automorphisms of $F_4^2$ in Tables \ref{tab:fibrations} and \ref{tab:16generators}.

\begin{center}
\renewcommand{\arraystretch}{1.0} 
\setlength{\tabcolsep}{1pt}
\begin{longtable}{|c|>{\small\centering\arraybackslash}m{3.5cm}|>{\small\centering\arraybackslash}m{2.5cm}||c|>{\small\centering\arraybackslash}m{3.8cm}|>{\small\centering\arraybackslash}m{2.5cm}|}
\caption{Elliptic fibrations on the Fermat quartic \(F_4^2\)} \label{tab:fibrations} \\
\hline
{\small No. $i$} & Dual graph of a reducible fiber \(Z_i\) & Section & {\small No. $i$} & Dual graph of a reducible fiber \(Z_i\) & Section\\
\hline
\endfirsthead
\hline
{\small No. $i$} & Dual graph of a reducible fiber \(Z_i\) & Section & {\small No. $i$} & Dual graph of a reducible fiber \(Z_i\) & Section\\
\hline
\endhead

3 & 
\begin{tikzpicture}[scale=0.45, every node/.style={circle,draw,line width=0.5pt,inner sep=1pt,minimum size=3.5pt}]
\node[label={[font=\tiny, inner sep=1pt]above left:\(N_{77}\)}] (a) at (0,0.7) {};
\node[label={[font=\tiny, inner sep=1pt]above right:\(M_{33}\)}] (b) at (0.7,0) {};
\node[label={[font=\tiny, inner sep=1pt]below right:\(L_{33}\)}] (c) at (0,-0.7) {};
\node[label={[font=\tiny, inner sep=1pt]below left:\(N_{71}\)}] (d) at (-0.7,0) {};
\draw (a)--(b)--(c)--(d)--(a);
\end{tikzpicture}
 & \(C_3=L_{71}\) \newline \((C_3.N_{77}=1)\) \newline \(C'_3=N_{37}\) \newline \((C_3'.N_{77}=1)\)& 
10 & 
\begin{tikzpicture}[scale=0.45, every node/.style={circle,draw,line width=0.5pt,inner sep=1pt,minimum size=3.5pt}]
\node[label={[font=\tiny, inner sep=1pt]below:\(N_{11}\)}] (N11) at (0,0) {};
\node[label={[font=\tiny, inner sep=1pt]above right:\(N_{17}\)}] (N17) at (1.2,0.8) {};
\node[label={[font=\tiny, inner sep=1pt]below right:\(N_{71}\)}] (N71) at (1.2,-0.8) {};
\node[label={[font=\tiny, inner sep=1pt]above:\(L_{53}\)}] (L53) at (-1.0,0) {};
\node[label={[font=\tiny, inner sep=1pt]above:\(M_{75}\)}] (M75) at (-2.2,0) {};
\node[label={[font=\tiny, inner sep=1pt]above right:\(L_{75}\)}] (L75) at (-3.6,1.0) {};
\node[label={[font=\tiny, inner sep=1pt]below right:\(M_{77}\)}] (M77) at (-3.6,-0.8) {};
\draw (N11)--(N17); \draw (N11)--(N71); \draw (N11)--(L53);
\draw (L53)--(M75);
\draw (M75)--(L75); \draw (M75)--(M77);
\end{tikzpicture}
 & \(C_{10}=L_{37}\) \newline \((C_{10}.N_{17}=1)\)\\
\hline

4 & 
\begin{tikzpicture}[scale=0.45, every node/.style={circle,draw,line width=0.5pt,inner sep=1pt,minimum size=3.5pt}]
\node[label={[font=\tiny, inner sep=1pt]above left:\(N_{53}\)}] (a) at (0,0.7) {};
\node[label={[font=\tiny, inner sep=1pt]above right:\(M_{17}\)}] (b) at (0.7,0) {};
\node[label={[font=\tiny, inner sep=1pt]below right:\(M_{37}\)}] (c) at (0,-0.7) {};
\node[label={[font=\tiny, inner sep=1pt]below left:\(L_{37}\)}] (d) at (-0.7,0) {};
\draw (a)--(b)--(c)--(d)--(a);
\end{tikzpicture}
 & \(C_4=L_{11}\) \newline \((C_4.N_{53}=1)\) & 
11 & 
\begin{tikzpicture}[scale=0.45, every node/.style={circle,draw,line width=0.5pt,inner sep=1pt,minimum size=3.5pt}]
\node[label={[font=\tiny, inner sep=1pt]above:\(L_{75}\)}] (L75) at (0,0) {};
\node[label={[font=\tiny, inner sep=1pt]above right:\(L_{77}\)}] (L77) at (1.2,0.8) {};
\node[label={[font=\tiny, inner sep=1pt]below right:\(N_{15}\)}] (N15) at (1.2,-0.8) {};
\node[label={[font=\tiny, inner sep=1pt]above:\(L_{55}\)}] (L55) at (-1.2,0) {};
\node[label={[font=\tiny, inner sep=1pt]above:\(L_{51}\)}] (L51) at (-2.4,0) {};
\node[label={[font=\tiny, inner sep=1pt]above:\(M_{37}\)}] (M37) at (-3.6,0) {};
\node[label={[font=\tiny, inner sep=1pt]above right:\(M_{35}\)}] (M35) at (-4.8,1.0) {};
\node[label={[font=\tiny, inner sep=1pt]below right:\(M_{57}\)}] (M57) at (-4.8,-0.8) {};
\draw (L75)--(L77); \draw (L75)--(N15); \draw (L75)--(L55);
\draw (L55)--(L51); \draw (L51)--(M37);
\draw (M37)--(M35); \draw (M37)--(M57);
\end{tikzpicture}
 & \(C_{11}=L_{17}\) \newline \((C_{11}.L_{77}=1)\)\\
\hline

5 & 
\begin{tikzpicture}[scale=0.45, every node/.style={circle,draw,line width=0.5pt,inner sep=1pt,minimum size=3.5pt}]
\node[label={[font=\tiny, inner sep=1pt]above left:\(L_{13}\)}] (c) at (0,0) {};
\node[label={[font=\tiny, inner sep=1pt]above right:\(L_{73}\)}] (L73) at (0,0.9) {};
\node[label={[font=\tiny, inner sep=1pt]above left:\(M_{73}\)}] (M73) at (0,2.0) {};
\node[label={[font=\tiny, inner sep=1pt]right:\(L_{17}\)}] (L17) at (1.2*0.866, -0.6) {};
\node[label={[font=\tiny, inner sep=1pt]below right:\(M_{31}\)}] (M31) at (2.0*0.866, -1.0) {};
\node[label={[font=\tiny, inner sep=1pt]left:\(M_{57}\)}] (M57) at (-1.2*0.866, -0.55) {};
\node[label={[font=\tiny, inner sep=1pt]below left:\(L_{35}\)}] (L35) at (-2.0*0.866, -1.1) {};
\draw (c)--(L73); \draw (L73)--(M73);
\draw (c)--(L17); \draw (L17)--(M31);
\draw (c)--(M57); \draw (M57)--(L35);
\end{tikzpicture}
 & \(C_5=L_{51}\) \newline \((C_5.M_{73}=1)\)  & 
12 & 
\begin{tikzpicture}[scale=0.45, every node/.style={circle,draw,line width=0.5pt,inner sep=1pt,minimum size=3.5pt}]
\node[label={[font=\tiny, inner sep=1pt]above:\(N_{13}\)}] (A) at (0,1) {};
\node[label={[font=\tiny, inner sep=1pt]above right:\(M_{31}\)}] (B) at (0.866,0.5) {};
\node[label={[font=\tiny, inner sep=1pt]below right:\(M_{35}\)}] (C) at (0.866,-0.5) {};
\node[label={[font=\tiny, inner sep=1pt]below:\(N_{71}\)}] (D) at (0,-1) {};
\node[label={[font=\tiny, inner sep=1pt]below left:\(M_{53}\)}] (E) at (-0.866,-0.5) {};
\node[label={[font=\tiny, inner sep=1pt]above left:\(M_{13}\)}] (F) at (-0.866,0.5) {};
\draw (A)--(B)--(C)--(D)--(E)--(F)--(A);
\end{tikzpicture}
 & \(C_{12}=N_{15}\) \newline \((C_{12}.N_{13}=1)\)\\
\hline

6 & 
\begin{tikzpicture}[scale=0.45, every node/.style={circle,draw,line width=0.5pt,inner sep=1pt,minimum size=3.5pt}]
\node[label={[font=\tiny, inner sep=1pt]above left:\(M_{33}\)}] (A) at (0,1) {};
\node[label={[font=\tiny, inner sep=1pt]above right:\(L_{11}\)}] (B) at (0.866,0.5) {};
\node[label={[font=\tiny, inner sep=1pt]below right:\(L_{13}\)}] (C) at (0.866,-0.5) {};
\node[label={[font=\tiny, inner sep=1pt]below right:\(M_{71}\)}] (D) at (0,-1) {};
\node[label={[font=\tiny, inner sep=1pt]below left:\(N_{35}\)}] (E) at (-0.866,-0.5) {};
\node[label={[font=\tiny, inner sep=1pt]above left:\(L_{77}\)}] (F) at (-0.866,0.5) {};
\draw (A)--(B)--(C)--(D)--(E)--(F)--(A);
\end{tikzpicture}
 & \(C_6=N_{77}\) \newline \((C_6.M_{33}=1)\) \newline \(C'_6=M_{37}\) \newline \((C'_6.M_{33}=1)\)& 
13 & 
\begin{tikzpicture}[scale=0.45, every node/.style={circle,draw,line width=0.5pt,inner sep=1pt,minimum size=3.5pt}]
\node[label={[font=\tiny, inner sep=1pt]right:\(M_{51}\)}] (v0) at (1.5, 0) {};
\node[label={[font=\tiny, inner sep=1pt]above right:\(N_{15}\)}] (v1) at (1.149, 0.964) {};
\node[label={[font=\tiny, inner sep=1pt]above:\(N_{55}\)}] (v2) at (0.260, 1.477) {};
\node[label={[font=\tiny, inner sep=1pt]above left:\(L_{71}\)}] (v3) at (-0.75, 1.299) {};
\node[label={[font=\tiny, inner sep=1pt]left:\(M_{35}\)}] (v4) at (-1.409, 0.513) {};
\node[label={[font=\tiny, inner sep=1pt]below left:\(M_{75}\)}] (v5) at (-1.409, -0.513) {};
\node[label={[font=\tiny, inner sep=1pt]below:\(N_{31}\)}] (v6) at (-0.75, -1.299) {};
\node[label={[font=\tiny, inner sep=1pt]below right:\(L_{55}\)}] (v7) at (0.260, -1.477) {};
\node[label={[font=\tiny, inner sep=1pt]right:\(L_{15}\)}] (v8) at (1.149, -0.964) {};
\draw (v0)--(v1)--(v2)--(v3)--(v4)--(v5)--(v6)--(v7)--(v8)--(v0);
\end{tikzpicture}
 & \(C_{13}=M_{53}\) \newline \((C_{13}.M_{51}=1)\)\\
\hline

7 & 
\begin{tikzpicture}[scale=0.45, every node/.style={circle,draw,line width=0.5pt,inner sep=1pt,minimum size=3.5pt}]
\node[label={[font=\tiny, inner sep=1pt]above:\(M_{71}\)}] (A) at (0,1.2) {};
\node[label={[font=\tiny, inner sep=1pt]above right:\(M_{31}\)}] (B) at (0.9,0.7) {};
\node[label={[font=\tiny, inner sep=1pt]above right:\(N_{13}\)}] (C) at (1.1,-0.3) {};
\node[label={[font=\tiny, inner sep=1pt]below right:\(N_{11}\)}] (D) at (0.5,-1.1) {};
\node[label={[font=\tiny, inner sep=1pt]below left:\(M_{55}\)}] (E) at (-0.5,-1.1) {};
\node[label={[font=\tiny, inner sep=1pt]below left:\(L_{55}\)}] (F) at (-1.1,-0.3) {};
\node[label={[font=\tiny, inner sep=1pt]above left:\(M_{77}\)}] (G) at (-0.9,0.7) {};
\draw (A)--(B)--(C)--(D)--(E)--(F)--(G)--(A);
\end{tikzpicture}
 & \(C_7=L_{13}\) \newline \((C_7.M_{71}=1)\)& 
14 & 
\begin{tikzpicture}[scale=0.45, every node/.style={circle,draw,line width=0.5pt,inner sep=1pt,minimum size=3.5pt}]
\node[label={[font=\tiny, inner sep=1pt]right:\(M_{37}\)}] (v0) at (1.5, 0) {};
\node[label={[font=\tiny, inner sep=1pt]above right:\(N_{55}\)}] (v1) at (1.149, 0.964) {};
\node[label={[font=\tiny, inner sep=1pt]above:\(N_{15}\)}] (v2) at (0.260, 1.477) {};
\node[label={[font=\tiny, inner sep=1pt]above left:\(L_{57}\)}] (v3) at (-0.75, 1.299) {};
\node[label={[font=\tiny, inner sep=1pt]left:\(M_{71}\)}] (v4) at (-1.409, 0.513) {};
\node[label={[font=\tiny, inner sep=1pt]below left:\(M_{75}\)}] (v5) at (-1.409, -0.513) {};
\node[label={[font=\tiny, inner sep=1pt]below:\(N_{31}\)}] (v6) at (-0.75, -1.299) {};
\node[label={[font=\tiny, inner sep=1pt]below right:\(L_{11}\)}] (v7) at (0.260, -1.477) {};
\node[label={[font=\tiny, inner sep=1pt]right:\(L_{15}\)}] (v8) at (1.149, -0.964) {};
\draw (v0)--(v1)--(v2)--(v3)--(v4)--(v5)--(v6)--(v7)--(v8)--(v0);
\end{tikzpicture}
 & \(C_{14}=M_{17}\) \newline\((C_{14}.M_{37}=1)\)\\
\hline

8 & 
\begin{tikzpicture}[scale=0.45, every node/.style={circle,draw,line width=0.5pt,inner sep=1pt,minimum size=3.5pt}]
\node[label={[font=\tiny, inner sep=1pt]above:\(L_{11}\)}] (L11) at (0,0) {};
\node[label={[font=\tiny, inner sep=1pt]above right:\(L_{13}\)}] (L13) at (1.2,0.8) {};
\node[label={[font=\tiny, inner sep=1pt]below right:\(M_{33}\)}] (M33) at (1.2,-0.8) {};
\node[label={[font=\tiny, inner sep=1pt]left:\(M_{77}\)}] (M77) at (-1.2,0) {};
\node[label={[font=\tiny, inner sep=1pt]above left:\(M_{17}\)}] (M17) at (-2.2,0.8) {};
\node[label={[font=\tiny, inner sep=1pt]below left:\(M_{75}\)}] (M75) at (-2.2,-0.8) {};
\draw (L11)--(L13); \draw (L11)--(M33);
\draw (L11)--(M77);
\draw (M77)--(M17); \draw (M77)--(M75);
\end{tikzpicture}
 & \(C_8=L_{73}\) \newline \((C_8.L_{13}=1)\)& 
15 & 
\begin{tikzpicture}[scale=0.45, every node/.style={circle,draw,line width=0.5pt,inner sep=1pt,minimum size=3.5pt}]
\node[label={[font=\tiny, inner sep=1pt]below:\(N_{75}\)}] (N75) at (0,0) {};
\node[label={[font=\tiny, inner sep=1pt]above right:\(M_{13}\)}] (M13) at (1.2,0.8) {};
\node[label={[font=\tiny, inner sep=1pt]below right:\(R_{1}\)}] (R1) at (1.2,-0.8) {};
\node[label={[font=\tiny, inner sep=1pt]above:\(M_{31}\)}] (M31) at (-1.2,0) {};
\node[label={[font=\tiny, inner sep=1pt]above:\(L_{75}\)}] (L75) at (-2.4,0) {};
\node[label={[font=\tiny, inner sep=1pt]above right:\(L_{15}\)}] (L15) at (-3.6,1.1) {};
\node[label={[font=\tiny, inner sep=1pt]below right:\(R_2\)}] (R2) at (-3.6,-0.8) {};
\draw (N75)--(M13); \draw (N75)--(R1); \draw (N75)--(M31);
\draw (M31)--(L75);
\draw (L75)--(L15); \draw (L75)--(R2);
\end{tikzpicture}
 & \(C_{15}=L_{57}\) \newline \((C_{15}.M_{13}=1)\)\\
\hline

9 & 
\begin{tikzpicture}[scale=0.45, every node/.style={circle,draw,line width=0.5pt,inner sep=1pt,minimum size=3.5pt}]
\node[label={[font=\tiny, inner sep=1pt]above:\(L_{35}\)}] (A) at (0,1) {};
\node[label={[font=\tiny, inner sep=1pt]above right:\(M_{13}\)}] (B) at (0.866,0.5) {};
\node[label={[font=\tiny, inner sep=1pt]below right:\(N_{57}\)}] (C) at (0.866,-0.5) {};
\node[label={[font=\tiny, inner sep=1pt]below:\(M_{75}\)}] (D) at (0,-1) {};
\node[label={[font=\tiny, inner sep=1pt]below left:\(L_{53}\)}] (E) at (-0.866,-0.5) {};
\node[label={[font=\tiny, inner sep=1pt]above left:\(N_{11}\)}] (F) at (-0.866,0.5) {};
\draw (A)--(B)--(C)--(D)--(E)--(F)--(A);
\end{tikzpicture}
 & \(C_9=L_{37}\) \newline \((C_9.L_{35}=1)\)& 
16 & 
\begin{tikzpicture}[scale=0.45, every node/.style={circle,draw,line width=0.5pt,inner sep=1pt,minimum size=3.5pt}]
\node[label={[font=\tiny, inner sep=1pt]below:\(M_{55}\)}] (M55) at (0,0) {};
\node[label={[font=\tiny, inner sep=1pt]above right:\(L_{55}\)}] (L55) at (1.2,0.8) {};
\node[label={[font=\tiny, inner sep=1pt]below right:\(R_{3}\)}] (R3) at (1.2,-0.8) {};
\node[label={[font=\tiny, inner sep=1pt]above:\(L_{33}\)}] (L33) at (-1.2,0) {};
\node[label={[font=\tiny, inner sep=1pt]above:\(N_{35}\)}] (N35) at (-2.4,0) {};
\node[label={[font=\tiny, inner sep=1pt]above right:\(N_{15}\)}] (N15) at (-3.6,1.1) {};
\node[label={[font=\tiny, inner sep=1pt]below right:\(R_{4}\)}] (R4) at (-3.6,-0.8) {};
\draw (M55)--(L55); \draw (M55)--(R3); \draw (M55)--(L33);
\draw (L33)--(N35);
\draw (N35)--(N15); \draw (N35)--(R4);
\end{tikzpicture}
 & \(C_{16}=N_{53}\) \newline \((C_{16}.L_{55}=1)\)\\
\hline
\end{longtable}
\end{center}

\begin{table}[h]
\centering
\renewcommand{\arraystretch}{1.3}
\caption{Automorphisms of the Fermat quartic \(F_4^2\)}\label{tab:16generators}
{\small
\begin{tabular}{|c|p{3.8cm}|c||c|p{2.8cm}|c|}
\hline
No. \( i\) & \text{Description of }\(g_i\) & Order & No. \(i\) & \text{Description of }\(g_i\) & Order \\
\hline
1 & \((x_0:x_1:x_2:x_3) \mapsto (\sqrt{-1}x_1:x_2:x_3:x_0)\) & 4 & 9 & \(\iota_{(Z_{9}, C_{9})}\) & 2 \\

2 & \((x_0:x_1:x_2:x_3) \mapsto (x_1:x_0:x_2:x_3)\) & 2 & 10 & \(\iota_{(Z_{10}, C_{10})}\) & 2 \\

3 & \(f_{(Z_3,C_3,C_3')} \circ \iota_{(Z_3, C_3)}\) & 2 & 11 & \(\iota_{(Z_{11}, C_{11})}\) & 2 \\

4 & \(\iota_{(Z_4, C_4)}\) & 2 & 12 & \(\iota_{(Z_{12}, C_{12})}\) & 2 \\

5 & \(\iota_{(Z_5, C_5)}\) & 2 & 13 & \(\iota_{(Z_{13}, C_{13})}\) & 2 \\

6 & \(f_{(Z_6,C_6,C_6')} \circ \iota_{(Z_6, C_6)}\) & 2 & 14 & \(\iota_{(Z_{14}, C_{14})}\) & 2 \\

7 & \(\iota_{(Z_7, C_7)}\) & 2 & 15 & \(\iota_{(Z_{15}, C_{15})}\) & 2 \\

8 & \(\iota_{(Z_8, C_8)}\) & 2 & 16 & \(\iota_{(Z_{16}, C_{16})}\) & 2 \\
\hline
\end{tabular}

\medskip

In this table, \(\iota_{(Z_{i},C_i)}\) (\(3\le i\le 16\)) is the inversion of the elliptic fibration \(\varphi_i\) of \(F_4^2\) with a reducible fiber \(Z_i\) and a zero section \(C_i\) given in Table \ref{tab:fibrations}. Moreover, \(f_{(Z_i,C_i,C'_i)}\) (\(i=3,6\)) is the translation automorphism of the elliptic fibration \(\varphi_i\) mapping \(C_i\) to \(C'_i\).
}
\end{table}

    Note that the finite subgroup ${\rm Aut}(F_4^2, h)= (\mu_{4}^{4} \rtimes \mathfrak{S}_{4})/\mu_{4}$ of ${\rm Aut}(F_4^2)$ preserving $h$ is generated by the two elements $g_1$ and $g_2$ in Table \ref{tab:16generators}. 
    
    \medskip
    
    Now we are ready to prove Theorem \ref{main1}.
    
    \medskip
    
   {\it Proof of Theorem \ref{main1}.}
  Let $S=S_3={\rm NS}(F_4^2)$. In this proof, we use computer to apply Dirichlet reduction. In the proof of Theorem \ref{main3}, we get an explicit set $\mathcal{S}$ of generators of ${\rm O}'(S)$ with $|\mathcal{S}|=19361$. By applying Dirichlet reduction to $\mathcal{S}$ and the hyperplane class $h$, we obtain a generating set $\mathcal{S}_1$ of ${\rm O}'(S)$ with $|\mathcal{S}_1|=28$. Here the method works quite effectively. Consider the natural maps $p_1: {\rm Aut}(F_4^2)\rightarrow {\rm O}'(S)$ and $p_S:{\rm O}(S)\rightarrow {\rm O}(A_{S},q_{A_S})$. Let $W\subset {\rm O}'(S)$ denote the Weyl group. Then $$W\rtimes {\rm Im}(p_1)=p_S^{-1}((p_S\circ p_1)({\rm Aut}(F_4^2))),$$ and it is a subgroup of ${\rm O}'(S)$ of index $4$. In fact, $p_S$ is surjective by \cite[Theorem 1.14.2]{Ni80}. Note that $|{\rm Im}(p_S)|=|{\rm Im}(p_S|_{{\rm O}'(S)})|$ since $S\cong U\oplus E_8^{\oplus 2}\oplus N$ and $$p_S(-{\rm id}_S)=p_S({\rm id}_{U\oplus E_8^{\oplus 2}}\oplus -{\rm id}_N)\in {\rm Im}(p_S|_{{\rm O}'(S)}),$$ where $N:=(-8)^{\oplus 2}$. Recall that $|{\rm Im}(p_S)|=|{\rm Im}(p_S|_{{\rm O}'(S)})|=|{\rm O}(A_S, q_{A_S})|=16$ and $$|(p_S\circ p_1)({\rm Aut}(F_4^2))|=4.$$ Then $[{\rm O}'(S): W\rtimes {\rm Im}(p_1)]=4$ and $W\rtimes {\rm Im}(p_1)=p_S^{-1}((p_S\circ p_1)({\rm Aut}(F_4^2)))$.
  
  Using \cite[Page 180, Theorem 6.9]{Su82} and $\mathcal{S}_1$, we find an explicit set $\mathcal{S}_2$ of generators of $W\rtimes {\rm Im}(p_1)$ with $|\mathcal{S}_2|=105$. Applying Dirichlet reduction again for $\mathcal{S}_2$ and $h$, we find a generating set $\mathcal{S}_3$ of $W\rtimes {\rm Im}(p_1)$ with $|\mathcal{S}_3|=17$. Sixteen elements in $\mathcal{S}_3$ can be expressed as words in $\mathcal{S}_4:=\{p_1(g_1),...,p_1(g_{16})\}$, and one is the $(-2)$-reflection $r_{[N_{73}]}$ with respect to the class $[N_{73}]$ of the line $N_{73}$.
This implies that $\mathcal{S}_4$ generates ${\rm Im}(p_1)$. Note that direct geometric realization of each element (a $20\times 20$ matrix) in $\mathcal{S}_3\setminus\{r_{[N_{73}]}\}$ seems complicated. To avoid this difficulty, we choose suitable geometric automorphisms, adding one by one if necessary,  which may replace $\mathcal{S}_3\setminus\{r_{[N_{73}]}\}$ as a generating set. The $16$ automorphisms $g_1, ..., g_{16}$ in Table \ref{tab:16generators} work nicely for this purpose. The explicit matrix forms of the isometries in $\mathcal{S}$ and $\mathcal{S}_i$ ($1\le i\le 4$), as well as the expressions of $\mathcal{S}$, $\mathcal{S}_2$, and $\mathcal{S}_3$ in terms of $\mathcal{S}_1$, $\mathcal{S}_3$, and $\mathcal{S}_4$ respectively, can be found in the ancillary file wordsofgenerators.txt.
  
 Since $p_1$ is injective, we conclude that $\{g_1,g_2,...,g_{16}\}$ is a generating set of ${\rm Aut}(F_4^2)$.\qed

\begin{remark}\label{rmk:order4}
Note that any set of generators of ${\rm Aut}(F_4^2)$ must contain at least one element of order $\ge 4$ since the image of the canonical representation ${\rm Aut}(F_4^2)\rightarrow {\rm GL}(H^0(F_4^2, \Omega_{F_4^2}^2))$ is of order $4$ (see the proof of Theorem \ref{main2}).
\end{remark}

\begin{remark}
Our three methods (overlattice structure, fundamental domains of finite-index subgroups and Dirichlet reduction) can be applied to automorphism groups of other K3 surfaces and birational automorphism groups of projective hyperk\"ahler manifolds.
\end{remark}

\end{document}